\documentclass{article}

\usepackage[
	a4paper,
	margin=25mm
]{geometry}
\usepackage{setspace}
\usepackage{sectsty} 
\usepackage{titlesec}
\usepackage{appendix}

\usepackage{amsmath, amssymb, amsthm}
\usepackage{mathtools}

\usepackage[
	setpagesize=false,
	colorlinks=true,
	linkcolor=blue,
        citecolor=blue,
        urlcolor=black,
	pdfencoding=auto,
	psdextra,
]{hyperref}
\usepackage{cleveref}

\usepackage{etoolbox} 
\usepackage[shortlabels]{enumitem}
\usepackage{xparse} 

\usepackage{graphicx}
\usepackage[dvipsnames]{xcolor}
\usepackage{tikz}

\usepackage{todonotes}
\usepackage[
    deletedmarkup=sout,
    commentmarkup=uwave,
    authormarkuptext=name
]{changes}

\usepackage{comment}

\setlist[enumerate,1]{label=(\arabic*), ref=(\arabic*)}
\setlist[enumerate,3]{label=(\roman*), ref=(\roman*)}

\allsectionsfont{\boldmath} 

\titlelabel{\thetitle.\quad}

\theoremstyle{plain}
\newtheorem{theorem}{Theorem}[section]
\newtheorem{lemma}[theorem]{Lemma}
\newtheorem{corollary}[theorem]{Corollary}
\newtheorem{proposition}[theorem]{Proposition}
\newtheorem{observation}[theorem]{Observation}
\newtheorem{conjecture}[theorem]{Conjecture}

\newtheorem{claim}{Claim}[section]
\newtheorem*{claim*}{Claim}

\makeatletter
\newenvironment{claimproof}[1][Proof]{\par
	\pushQED{\qed}%
	
	\normalfont \topsep6\p@\@plus6\p@\relax
	\trivlist
	\item[\hskip\labelsep
	\textit{#1}\@addpunct{.}~]\ignorespaces
}{%
	\popQED\endtrivlist\@endpefalse
}
\makeatother

\newlist{Cases}{enumerate}{3}
\setlist[Cases]{parsep=0pt plus 1pt}
\setlist[Cases,1]{wide=0pt, listparindent=\parindent,
    label = \textbf{Case~\arabic*:}, ref = \arabic*}
\setlist[Cases,2]{wide=\parindent, listparindent=\parindent,
    label = \textbf{Case~\arabic{Casesi}-\arabic{Casesii}:}}

\crefname{Casesi}{case}{cases}

\newcounter{case}
\AtBeginEnvironment{proof}{\setcounter{case}{0}}

\crefname{case}{case}{cases}

\theoremstyle{definition}
\newtheorem{definition}[theorem]{Definition}

\newcommand{\calP}{\mathcal{P}}

\newcommand{\ve}{\varepsilon}

\makeatletter
\NewDocumentCommand{\xsideset}{mmme{_^}}{%
  \mathop{%
    \settowidth{\dimen0}{$\m@th\displaystyle#3$}%
    \dimen0=.5\dimen0
    \settowidth{\dimen2}{$%
      \m@th\displaystyle#3%
      \IfValueT{#4}{_{#4}}%
      \IfValueT{#5}{^{#5}}%
    $}%
    \dimen2=.5\dimen2
    \advance\dimen2 -\dimen0
    \sbox6{\scriptspace\z@$\displaystyle{\vphantom{#3}}#1$}
    \sbox8{\scriptspace\z@$\displaystyle{\vphantom{#3}}#2$}
    \ifdim\wd6>\dimen2 \kern\dimexpr\wd6-\dimen2\relax\fi
    {%
     \mathop{\llap{\copy6}{\displaystyle#3}\rlap{\copy8}}\limits
     \IfValueT{#4}{_{#4}}%
     \IfValueT{#5}{^{#5}}%
    }%
    \ifdim\wd8>\dimen2 \kern\dimexpr\wd8-\dimen2\relax\fi
  }%
}
\makeatother

\newcommand{\defeq}{\coloneqq}

\let\originalleft\left
\let\originalright\right
\renewcommand{\left}{\mathopen{}\mathclose\bgroup\originalleft}
\renewcommand{\right}{\aftergroup\egroup\originalright}

\makeatletter
\newcases{lrcases*}
  {\quad}
  {$\m@th{##}$\hfil}
  {{##}\hfil}
  {\lbrace}
  {\rbrace}
\makeatother

\usetikzlibrary{matrix,arrows,decorations.pathmorphing}
\usetikzlibrary{positioning}
\usetikzlibrary{graphs} 
\usetikzlibrary{graphs.standard}

\title{Random matchings in linear hypergraphs}

\author{
Hyunwoo Lee%
        \thanks{Department of Mathematical Sciences, KAIST, South Korea and Extremal Combinatorics and Probability Group (ECOPRO), Institute for Basic Science (IBS).
        E-mail: {\ttfamily hyunwoo.lee@kaist.ac.kr.} Supported by the National Research Foundation of Korea (NRF) grant funded by the Korea government(MSIT) No. RS-2023-00210430, the National Research Foundation of Korea (NRF) grant funded by the Korea government(MSIT) No. 2022R1C1C1010300, and the Institute for Basic Science (IBS-R029-C4).}
}

\usepackage[
    backend=biber,
    sorting=nyt,
    maxbibnames=99
]{biblatex}
\begin{document}
\maketitle

\begin{abstract}
    For a given hypergraph $H$ and a vertex $v\in V(H)$, consider a random matching $M$ chosen uniformly from the set of all matchings in $H.$ In $1995,$ Kahn conjectured that if $H$ is a $d$-regular linear $k$-uniform hypergraph, the probability that $M$ does not cover $v$ is $(1 + o_d(1))d^{-1/k}$ for all vertices $v\in V(H).$ This conjecture was proved for $k = 2$ by Kahn and Kim in $1998.$
    
    In this paper, we disprove this conjecture for all $k \geq 3.$ For infinitely many values of $d,$ we construct $d$-regular linear $k$-uniform hypergraph $H$ containing two vertices $v_1$ and $v_2$ such that $\calP(v_1 \notin M) = 1 - \frac{(1 + o_d(1))}{d^{k-2}}$ and $\calP(v_2 \notin M) = \frac{(1 + o_d(1))}{d+1}.$ The gap between $\calP(v_1 \notin M)$ and $\calP(v_2 \notin M)$ in this $H$ is best possible. In the course of proving this, we also prove
    a hypergraph analog of Godsil's result on matching polynomials and paths in graphs, which is of independent interest.
\end{abstract}


\section{Introduction}\label{sec:intro}

A matching in a $k$-graph ($k$-uniform hypergraph) $H$ is a set of vertex-disjoint edges of $H.$ Matching is one of the most natural objects in combinatorics, hence there are various topics related to matchings and each of them has been extensively studied. Especially, determining the maximum or average size of matchings in a given (hyper)graph is a central topic in the matching theory. For instance, Erd\H{o}s matching conjecture~\cite{erdos-matching} is one of the most important conjectures in the extremal set theory, which regards the maximum sizes of matchings in given set systems.

Among various hypergraphs, (almost) regular hypergraphs form an interesting class of hypergraphs. For instance, Vizing's theorem~\cite{vizing} implies that for any $n$-vertex $d$-regular graph have a matching of size at least $\left(1 - \frac{1}{d+1} \right)\frac{n}{2}$, thus an almost perfect matching if $d$ is large.

For hypergraphs, Pippenger(unpublished), and Pippenger, Spencer~\cite{pippenger-spencer} proved that almost regular $k$-graphs with small codegrees have almost perfect matchings. This fact and its generalizations have numerous useful applications in design theory~\cite{keevash2014existence,design-iterative}, hypergraph coloring problems~\cite{efl}, Latin squares~\cite{montgomery2023proof}, extremal combinatorics~\cite{triangle-packing,n-queen}, and so on.

Another natural question is regarding the average size of matchings in the regular linear hypergraph. Let $H$ be a $d$-regular linear $k$-graph and $M$ be a matching chosen uniformly at random from the set of all matchings of $H.$
For vertices $v_1, \dots, v_s, u_1,\dots, u_t\in V(H)$, denote by $\calP_H(\overline{v_1}, \dots, \overline{v_s}| \overline{u_1}, \dots, \overline{u_t})$ the probability that $M$ does not cover any of $v_1, \dots v_s$ conditioning on the event that $M$ does not cover any of $u_1, \dots, u_s.$ In particular, $\calP_H(\overline{v_1}, \dots, \overline{v_s})$ denote the probability that $M$ does not cover any of $v_1, \dots v_s.$ Then the average size of matchings in $H$ is closely related to the values of $\calP_H(\overline{v}).$

For distinct vertices $u$ and $v$, the linearity of $H$ implies that the set of edges containing $u$ is disjoint from the set of edges containing $v$ possibly except one. Moreover, $d$-regularity ensures that these sets have the same size. From these facts, it is reasonable to suspect that $u$ and $v$ are indistinguishable in terms of the number of matchings containing them. Indeed, Kahn~\cite{kahn1995,kahn-1997} proposed the following conjecture stating that this suspicion is true.

\begin{conjecture}[Kahn~\cite{kahn1995,kahn-1997}]\label{conj:kahn}
    Let $H$ be a $d$-regular linear $k$-graph. Then the following holds for all $v \in V(H).$
    $$\calP_H(\overline{v}) = (1 + o_d(1))d^{-1/k}.$$
\end{conjecture}

To understand the meaning of the number $d^{-1/k}$, assume $\calP_H(\overline{v}) = x$ for all $v\in V(H).$
Observe that for each edge $e = \{v_1, \dots, v_k\}\in E(H)$, we have $\calP_H(e\in M) = \calP_H(\overline{v_1}, \dots, \overline{v_k}).$
For a vertex $v\in V(H),$ let $e_1, \dots, e_d$ be the edges of $H$ containing $v$, where $e_i = \{v, v_{(i, 1)}, \dots, v_{(i, k-1)}\}$ for each $i\in [d].$ 
Then we have
$$\calP_H(\overline{v}) + \sum_{v\in e\in E(H)} \calP_H(e\in M) = \calP_H(\overline{v}) + \sum_{i\in [d]} \calP_H\left(\overline{v}, \overline{v_{(i, 1)}}, \dots, \overline{v_{(i, k-1)}}\right) = 1.$$

By dividing $\calP_H(\overline{v})$ on both sides, we obtain 
$$1 + \sum_{i\in [d]} \calP_H\left(\overline{v_{(i, 1)}}, \dots, \overline{v_{(i, k-1)}} \middle| \overline{v}\right) = \calP_H(\overline{v})^{-1}.$$

We note that $\calP_H\left(\overline{v_{(i, 1)}}, \dots, \overline{v_{(i, k-1)}} \middle| \overline{v}\right) = \calP_{H - v}\left(\overline{v_{(i, 1)}}, \dots, \overline{v_{(i, k-1)}}\right).$ As the linearity of $H$ implies that deleting $v$ does not change the set of edges containing $v_{(i, j)}$ much, it is reasonable to suspect $\calP_H\left(\overline{v_{(i, 1)}}, \dots, \overline{v_{(i, k-1)}} \middle| \overline{v}\right) = x^{k-1}.$
From this, we obtain the equality, 
$x = \frac{1}{1 + d x^{k-1}},$ which has the unique positive solution $(1 + o_d(1))d^{-1/k}.$ This provides a heuristic justification for the value $d^{-1/k}$ in \Cref{conj:kahn}.

Indeed, Kahn and Kim~\cite{kahn-kim} verified \Cref{conj:kahn} for graphs (which are linear $2$-graphs) and also provided the asymptotically correct value of the variance of the size of random matchings.

\begin{theorem}[Kahn and Kim~\cite{kahn-kim}]\label{thm:kahn-kim}
    Let $G$ be an $n$-vertex $d$-regular graph. Let $M$ be a random matching chosen uniformly at random from the set of all matchings of $G.$ Then we have the following for all $v\in V(G).$

    \begin{enumerate}
        \item[$(a)$] $\calP_G(\overline{v}) = (1 + o_d(1))d^{-1/2}$, so the expectation of $|M|$ is $\left(1 - \frac{1 + o_d(1)}{\sqrt{d}} \right)\frac{n}{2}$,
        \item[$(b)$] the variance of $|M|$ is $(1 + o_d(1))\frac{n}{4\sqrt{d}}.$
    \end{enumerate}
\end{theorem}

They further conjectured the following, which is stronger than \Cref{conj:kahn}.
\begin{conjecture}[Kahn and Kim~\cite{kahn-kim}]\label{conj:kahn-kim}
    Let $k \geq 2$ be an integer and $H$ be an $n$-vertex $d$-regular linear $k$-graph. Let $M$ be a random matching chosen uniformly at random from the set of all matchings of $H$ Then we have the following for all $v\in V(H).$

    \begin{enumerate}
        \item[$(a)$] $\calP_G(\overline{v}) = (1 + o_d(1))d^{-1/k},$
        \item[$(b)$] the variance of $|M|$ is $(1 + o_d(1))\frac{n}{k^2 d^{1/k}}.$
    \end{enumerate}
\end{conjecture}

\Cref{thm:kahn-kim} implies that the statistics for the random matchings of large regular graphs hardly depend on their specific structures. Surprisingly, this is no longer true for hypergraphs.

\begin{theorem}\label{thm:main}
    Let $k > 2$ be an integer and $d_0, \ve > 0$ be a positive real number. Then there exist $d > d_0$ and $d$-regular linear $k$-graph $H$ such that the following holds. There are two vertices $v_1, v_2 \in V(H)$ which satisfies the following.
    \begin{enumerate}
        \item[$\bullet$] $\calP_H(\overline{v_1}) > 1 - \frac{1 + \ve}{d^{k - 2}},$
        \item[$\bullet$] $\calP_H(\overline{v_2}) < \frac{1 + \ve}{d+1}.$
    \end{enumerate}
\end{theorem}

Not only this disproves \Cref{conj:kahn} (as well as \Cref{conj:kahn-kim}), it shows that the gap between $\calP_H(\overline{v_1})$ and $\calP_H(\overline{v_2})$ can be large, one being $1 - o_d(1)$ and the other being $o_d(1).$ Indeed, the following theorem determines the correct range of $\calP_H(\overline{v})$, showing that \Cref{thm:main} is best possible.

\begin{theorem}\label{thm:main-tight}
    Let $k > 2$ be an integer and $\ve > 0$ be a real number. Then there exists $d_0 > 0$ such that for all $d \geq d_0$ and for all $d$-regular linear $k$-graph $H$, the following inequality holds for all $v\in V(H).$
    $$\frac{1}{d+1} \leq \calP_H(\overline{v}) < 1 - \frac{1 - \ve}{d^{k-2}}.$$
\end{theorem}

To prove \Cref{thm:main-tight}, we extend the celebrated result of Godsil~\cite{godsil} on matching polynomials and paths in graphs in \Cref{sec:define-tree}. One direct corollary of Godsil's result is for every graph $G$ and $v \in V(G),$ there is a tree $T$ and a vertex $u \in V(T)$ such that $\calP_G(\overline{v}) = \calP_T(\overline{u}).$
This means that we can count on $\calP_T(\overline{u})$ instead of $\calP_G(\overline{v})$ and the structure of $T$ allows us to compute $\calP_T(\overline{u})$ in a recursive manner. This was the main idea of the proof of \Cref{thm:kahn-kim}. We generalize this result into hypergraphs. This would be an independent interest.

\paragraph{Organization.}
We prove \Cref{thm:main} in \Cref{sec:construction} by explicitly constructing the regular linear $k$-graph $H$ that satisfies the conclusion of \Cref{thm:main}. In \Cref{sec:define-tree}, we extend Godsil~\cite{godsil}'s result to hypergraph by defining a new type of walk on hypergraph, that we refer to as \emph{conflict-free walk}. Using this result, we prove \Cref{thm:main-tight} in \Cref{sec:proof-main-tight}. Lastly, we discuss several concluding remarks in \Cref{sec:concluding}.


\section{Random matchings and dynamical systems}\label{sec:construction}

\subsection{$d$-extendable linear hypergraphs}
The idea to prove \Cref{thm:main} is to define an operation that can inductively construct a hypertree-like structure while keeping the degree condition. For this, we introduce some definitions.
We say a $k$-graph $H$ is \emph{$d$-extendable} if all vertices of $H$ has degree $d$ except one vertex $v$ of degree $d-1.$ We say that the vertex $v$ of degree $d-1$ is the \emph{head of $H$} and write it as $\mathrm{head}(H).$ The following lemma shows the existence of $d$-extendable linear $k$-graph.

\begin{lemma}\label{lem:base-existence}
    Let $k \geq 2$ and $\ell \geq 0$ be non-negative integers. Then there is a $(k\ell + 1)$-extendable linear $k$-graph $H.$ 
\end{lemma}

\begin{proof}
    We use induction on $\ell.$
    The base case, when $\ell = 0$, is trivial by taking a single edge with one isolated vertex. Denote this $k$-graph as $H_0.$ Now we assume we have a $(k(\ell - 1) + 1)$-extendable linear $k$-graph $H_{\ell - 1}.$

    \begin{claim}\label{clm:regular-existence}
        For every $d \geq 1$, there is a $d$-regular linear $k$-graph.
    \end{claim}

    \begin{claimproof}[Proof of \Cref{clm:regular-existence}]
        We use induction on $d.$ If $d = 1$, we take a single edge, then this is $1$-regular linear hypergraph. Now assume there is a $(d-1)$-regular linear $k$-graph $F.$ Consider $F_1, \dots, F_k$ disjoint copies of $F.$ We add a perfect matching on $V(F_1) \cup \cdots \cup V(F_k)$ where each edge of the perfect matching contains exactly one vertex of $F_i$ for each $i\in [k].$ Then the obtained hypergraph is $d$-regular and still a linear hypergraph. This proves the claim.
    \end{claimproof}

    By \Cref{clm:regular-existence}, we can take a $(k\ell + 1)$-regular linear $k$-graph $F.$
    Let $F_1, \dots, F_{k-1}$ be disjoint copies of $F.$ We choose an edge $e_i = \{v_{(i, 1)}, \dots, v_{(i, k)}\}$ of $F_i$ for each $i\in [k-1].$ We introduce a new vertex $v'$ that is disjoint from $\bigcup_{i\in [k-1]}V(F_i)$ and remove $e_i$ from $F_i$ for each $i\in [k-1].$ Now we add a new hyperedge $\{v', v_{(i, 1)}, \dots, v_{(i, k-1)}\}$ for each $i\in [k-1]$ and we finally add a hyperedge $\{v', v_{(1, k)}, \dots, v_{(k-1, k)}\}.$ Let us write this hypergraph as $H'_{\ell}.$ We note that degree of $v'$ in $H'_{\ell}$ is $k$ and the other vertices have degree $k\ell + 1.$
    Let $H''_{\ell}$ be a $k$-graph which is disjoint union of $|V(H_{\ell})|$ copies of $H'_{\ell - 1}.$ Let $V$ be the set of vertices whose degrees are $k.$ We define $H_{\ell}$ as a union of $H''_{\ell}$ and $H_{\ell - 1}$ on the vertex set $V.$ Then $H_{\ell}$ is a linear $k$-graph such that one vertex have degree $k\ell$ and the others have degree $k\ell + 1.$ This completes the proof. 
\end{proof}

We now define an operation that extends a $d$-extendable linear $k$-graph into a larger $d$-extendable linear $k$-graph as follows.

\begin{definition}
    Let $F$ be a $d$-extendable linear $k$-graph. Let $F'$ be a disjoint union of copies $F_{(i, j)}$ of $F$ where $i\in [d-1]$ and $j\in [k-1].$ We introduce a new vertex $v$ which is disjoint from $V(F').$ We now define a $k$-graph $S_d(F)$ as $$V(S_d(F)) = V(F')\cup \{v\}$$ and $$E(S_d(F)) = E(F')\cup \bigcup_{i\in [d-1]} \{v, \mathrm{head}(F_{(i, 1)}), \dots, \mathrm{head}(F_{(i, k-1)})\}.$$ Let $S_d^{(0)} = F$ and for non-negative integer $\ell$, we denote by $S_d^{(\ell)}(F)$ the $k$-graph $S_d(S_d^{(\ell-1)}(F)).$
\end{definition}

We note that if $F$ is a $d$-extendable linear $k$-graph, then so is $S_d(F)$ and the newly introduced vertex $v$ in the definition of $S_d(\cdot)$ becomes a new head. Thus $S_d^{(\ell)}(F)$ is well-defined.
We will use the operation $S_d(\cdot)$ iteratively to prove \Cref{thm:main}. To do this, we analyze several numerical properties that are related to $S_d(\cdot).$ 


\subsection{Attractors on extreme points}

From now on, we consider $k > 2$ to be a fixed integer. For all integer $d\geq 2$, we define a function $g_{d}(x)$ and $f_{d}(x)$ as follows.
$$g_{d}(x) = \frac{1}{1 + (d-1)x^{k-1}}$$ and $$f_{d}(x) = g_{d}(g_{d}(x)) = \left(1 + \frac{d-1}{(1 + (d-1)x^{k-1})^{k-1}} \right)^{-1}.$$
As $q(x) = x(g_{d}(x))^{-1}$ is an increasing function with $q(0) < 1 < q(1),$ there is unique solution $0 < \alpha_{d} < 1$ such that $\alpha_{d} = g_{d}(\alpha_{d}).$ It is routine to check $\alpha_{d} = (1 - o_d(1))d^{-1/k}.$
The purpose of this section is to show the dynamical system $x_{\ell + 1} = f_d(x_{\ell})$ starting from $x_0 \in [0, 1]$ has two attractors one in $(0, \alpha_d)$ and another in $(\alpha_d, 1).$

\begin{lemma}\label{lem:f-three-solution}
    For sufficiently large $d$, the equation $x = f_d(x)$ defined on the interval $[0, 1]$ has exactly three solutions $0 < \gamma_d < \alpha_d < \beta_d < 1$.
\end{lemma}

\begin{proof}[Proof of \Cref{lem:f-three-solution}]
    We first show that the equation $x = f_d(x)$ has at most one solution on the interval $(\alpha_d, 1].$
    \begin{claim}\label{clm:num-1}
        For all $x$ in the interval $(\alpha_d, \frac{1}{2}],$ the inequality $f_d(x) > x$ holds.
    \end{claim}
    
    \begin{claimproof}[Proof of \Cref{clm:num-1}]
        Let $x = \alpha_d \mu$, where $1 < \mu \leq \frac{1}{2\alpha_d}.$ Then we have 
        \allowdisplaybreaks
        \begin{align*}
            g_d(x) &= \frac{1}{1 + (d-1)\alpha_d^{k-1}\mu^{k-1}} = \frac{1}{1 + (d-1)\alpha_d^{k-1}} \times \frac{1 + (d-1)\alpha_d^{k-1}}{1 + (d-1)\alpha_d^{k-1}\mu^{k-1}}\\
            &= \alpha_d \times \frac{1/\alpha_d}{1 + (1 / \alpha_d - 1)\mu^{k-1}} = \alpha_d \left(\frac{1}{\alpha_d + (1 - \alpha_d)\mu^{k-1}}\right).
        \end{align*}
        Since $f_d(x) = g_d(g_d(x))$, we have $f_d(\alpha_d \mu) = \alpha_d\left(\alpha_d + \frac{(1 - \alpha_d)}{(\alpha_d + (1 - \alpha_d)\mu^{k-1})^{k-1}} \right)^{-1}.$ To complete the proof, it suffices to show that the following inequality holds. $$\alpha_d \mu + \frac{(1 - \alpha_d) \mu}{(\alpha_d + (1 - \alpha_d)\mu^{k-1})^{k-1}} < 1.$$
        Let $h(\mu) \defeq \alpha_d \mu + \frac{(1 - \alpha_d) \mu}{(\alpha_d + (1 - \alpha_d)\mu^{k-1})^{k-1}}.$ Then for all $1 \leq \mu \leq 2$, we have
        \begin{align*}
            h'(\mu) &= \alpha_d + \frac{(1 - \alpha_d)(\alpha_d + (1 - \alpha_d)\mu^{k-1})^{k-1} - (1 - \alpha_d)^2(k-1)^2(\alpha_d + (1 - \alpha_d)\mu^{k-1})^{k-2}\mu^{k-1}}{(\alpha_d + (1 - \alpha_d)\mu^{k-1})^{2k-2}}\\
            &= \alpha_d + \frac{(1 - \alpha_d)(\alpha_d + (1 - \alpha_d)\mu^{k-1}) - (1 - \alpha_d)^2(k-1)^2 \mu^{k-1}}{(\alpha_d + (1 - \alpha_d)\mu^{k-1})^{k}}\\
            &\leq \frac{2^{k^2}\alpha_d + (1 - \alpha_d)(\alpha_d + (1 - \alpha_d)\mu^{k-1}) - (1 - \alpha_d)^2(k-1)^2 \mu^{k-1}}{(\alpha_d + (1 - \alpha_d)\mu^{k-1})^{k}}\\
            &\leq \frac{(1 + 2^{k^2}\alpha_d - (1 - \alpha_d)^2(k-1)^2)\mu^{k-1}}{(\alpha_d + (1 - \alpha_d)\mu^{k-1})^{k}}\\
            &< 0.
        \end{align*}
        For the last inequality, we use the condition $k > 2.$ Since $h(1) = 1 \leq 1$, for all $\mu\in (1, 2]$, we have $h(\mu) < 1.$
        Now, we consider the case when $\mu \in (2, \frac{1}{2\alpha_d}].$
        Then the following inequality holds.
        \begin{align*}
            h(\mu) \leq \alpha_d \frac{1}{2\alpha_d} + \frac{\mu}{((1 - \alpha_d)\mu^{k-1})^{k-1}} \leq \frac{1}{2} + \frac{1}{8(1 - \alpha_d)^{k-1}} < 1.
        \end{align*}
        This proves the claim.
    \end{claimproof}

    \begin{claim}\label{clm:num-2}
        For all $x\in (\frac{1}{2}, 1)$, we have $f_d'(x) < 1.$
    \end{claim}

    \begin{claimproof}[Proof of \Cref{clm:num-2}]
        By computing the derivative directly, we have the following inequalities.
        \begin{align*}
            f_d'(x) &= \frac{(1 + (d-1)x^{k-1})^{2k-2}}{((d - 1) + (1 + (d-1)x^{k-1})^{k-1})^2}\times \frac{(d-1)^2(k-1)^2(1 + (d-1)x^{k-1})^{k-2}x^{k-2}}{(1 + (d-1)x^{k-1})^{2k-2}}\\
            &= (d-1)^2(k-1)^2 \times \frac{(1 + (d-1)x^{k-1})^{k-2}x^{k-2}}{((d - 1) + (1 + (d-1)x^{k-1})^{k-1})^2}\\
            &\leq \frac{(k-1)^2 d^k}{((d - 1) + (1 + (d-1)x^{k-1})^{k-1})^2}\\
            &\leq \frac{(k-1)^2 d^k}{2^{-2(k-1)^2} (d-1)^{2k-2}}\\
            &< 1.
        \end{align*}
        This proves the claim.
    \end{claimproof}

    Assume we have two distinct solutions $x_1 < x_2$ for the equation $x = f_d(x)$, where $x_1, x_2 \in (\alpha_d, 1].$ We note that $1 \neq f_d(1)$, so nither $x_1$ nor $x_2$ are $1.$ By \Cref{clm:num-1}, we have $\frac{1}{2} < x_1 < x_2 < 1.$ By \Cref{clm:num-2}, the function $f_d(x) - x$ is strictly decreasing when $x > \frac{1}{2}.$ This implies there are no two distinct solutions in the interval $(\alpha_d, 1]$, a contradiction. Thus there is at most one solution in $(\alpha_d, 1).$ Since $f_d(1) < 1$ and $f_d(\frac{1}{2}) > \frac{1}{2}$, there is a unique solution $\beta_d \in (\alpha_d , 1]$ by the intermediate value theorem.

    Now, let $\gamma_d \defeq g_d(\beta_d).$ Then it is easy to see that $\gamma_d$ is also a solution of the equation $x = f_d(x).$ As $\alpha_d$ is the unique solution of $g_d(x) = x$, $\gamma_d$ is distinct from both $\alpha_d$ and $\gamma_d.$ Since $g_d$ is decrasing function, and $\beta_d > \alpha_d$, this implies $0 < \gamma_d < \alpha_d.$ 

    Assume there is another solution $\gamma'_d \in (0, \alpha_d)$ of the equation $x = f_d(x)$, where $\gamma'_d$ is distinct from $\gamma_d.$ Then $g_d(\gamma'_d)$ also be a solution of $x = f_d(x)$ and since $g_d$ is a decreasing function, we have $\beta_d \neq g_d(\gamma'_d) \in (\alpha_d, 1).$ This contradict to the fact that $\beta_d$ is the unique solution of $x = f_d(x).$ This completes the proof.
\end{proof}

From \cref{lem:f-three-solution} and the intermediate value theorem, it is easy to check the following.

\begin{observation}\label{obs:attractors-position}
    For all sufficiently large $d$, we have
    $\beta_d = 1 - \frac{1 + o_d(1)}{d^{k-2}}$, and $\gamma_d = \frac{1 + o_d(1)}{d + 1}.$ For each $x\in [0, 1]$, the following holds.
    $$
    \begin{cases}
        f_d(x) > x &\text{if } x\in [0, \gamma_d),\\
        f_d(x) < x &\text{if } x\in (\gamma_d, \alpha_d),\\
        f_d(x) > x &\text{if } x\in (\alpha_d, \beta_d),\\
        f_d(x) < x &\text{if } x\in (\beta_d, 1],\\
    \end{cases}
    $$
\end{observation}

The next proposition states that if a strictly increasing function $h$ and a point $p$ with $p = h(p)$ have an interval $I$ containing $p$ which satisfies $h(a) > a$ and $h(b) < b$ for all $a < p < b$ and $a, b \in I$, then the dynamical system $p_{i + 1} = h(p_{i})$ start from a point $p_0 \in I$ eventually converges to $p.$

\begin{proposition}\label{prop:dynamical}
    Let $h$ be a strictly increasing real-valued continuous function that is defined on an interval $I$ and there is a point $p\in I$ such that $p = h(p).$ Assume for all $a, b \in I$ such that $a < p < b$ satisfies $h(a) > a$ and $h(b) < b.$ Let $p_0 \in I$ be an arbitrary point and define $p_{i+1} = h(p_{i})$ inductively for each $i \in \mathbb{Z}_{\geq 0}.$ Then the following holds. $$\lim_{i \to \infty} p_{i} = p.$$
\end{proposition}

\begin{proof}[Proof of \Cref{prop:dynamical}]
    Without loss of generality, assume $p_0 < p.$ Since $h$ is a strictly incrasing function and $h(a) > a$ for all $a < p$ in the interval $I$, the sequence $\{p_i\}_{i\in \mathbb{Z}_{\geq 0}}$ is a strictly increasing function.
    We claim that the sequnce $\{p_{i}\}_{i\in \mathbb{Z}_{\geq 0}}$ is bounded above by $p.$ Assume this is not true. Let $\ell$ be the non-negative integer such that $p_{\ell} > p.$ We note that $p_0 < p$, so $\ell > 1.$ Then $p_{\ell} = h(p_{\ell - 1})$ and $p_{\ell - 1 } < p.$ Since $h$ is strictly increasing, this implies $p_{\ell} = h(p_{\ell - 1}) < h(p) = p$, a contradiction. 

    Thus the sequence $\{p_{i}\}_{i\in \mathbb{Z}_{\geq 0}}$ is bounded above by $p$ and it is strictly increasing function, the limit $\lim_{i \to \infty} p_{i}$ exists and its value is bounded above by $p.$
    
    Assume $p' < p$ be the limit value of the sequence $\{p_{i}\}_{i \in \mathbb{Z}_{\geq 0}}.$ Note that as $h$ is a strictly incrasing function, its inverse function $h^{-1}$ exists. Since $p' < p$, we have $h^{-1}(p') < p'.$ By the definition of the limits of sequences, there exists $\ell$ such that $h^{-1}(p') < p_{\ell}.$ Then we have $p_{\ell + 1} > p'$ since $h$ is a strictly increasing function. Moreover, the sequence $\{p_i\}_{i\in \mathbb{Z}_{\geq 0}}$ is a strictly increasing function, this contradicts the assumption that the limit of the sequence is $p'.$ This implies $\lim_{i\to \infty} p_i = p.$ This completes the proof.
\end{proof}

From \Cref{obs:attractors-position} and \Cref{prop:dynamical}, we directly obtain the following lemma.

\begin{lemma}\label{lem:attractor}
    Let $p_0 \in [0, 1]$ and recursively define $p_{i+1} = f_d(p_i)$ for each $i\in \mathbb{Z}_{\geq 0}.$ Then the following holds.
    $$
    \lim_{i\to \infty} p_i =
    \begin{cases}
        \beta_d &\text{ if } \alpha_d < p_0 \leq 1,\\
        \gamma_d &\text{ if } 0 \leq p_0 < \alpha_d.
    \end{cases}
    $$
\end{lemma}


\subsection{Proof of \Cref{thm:main}}

In this section, we prove the following lemma. \Cref{thm:main} will be deduced from this lemma.

\begin{lemma}\label{lem:counterexample}
    For all sufficiently large $d \equiv 1 \pmod{k},$ there is a $d$-extendable linear $k$-graph $F_0$ such that the following holds.
    \begin{enumerate}
        \item[$\bullet$] $\calP_{S_d^{(2\ell)}(F_0)}\left(\overline{\mathrm{head}(S_d^{(2\ell)}(F_0))}\right) = (1 + o_{\ell}(1))\beta_d,$
        \item[$\bullet$] $\calP_{S_d^{(2\ell - 1)}(F_0)}\left(\overline{\mathrm{head}(S_d^{(2\ell - 1)}(F_0))}\right) = (1 + o_{\ell}(1))\gamma_d.$
    \end{enumerate}
\end{lemma}

We use the following simple observation crucially in the proof \Cref{lem:counterexample}

\begin{observation}\label{obs:independent}
    Let $H$ be a disjoint union of $k$-graphs $H_1, H_2, \dots H_m.$ For each $i\in [m]$, let $v_i \in V(H_i)$ be arbitrary vertices. Then for the uniformly chosen random matching $M$ of $H,$ the events $v_1\notin M, \dots, v_m\notin M$ are mutually independent. Moreover, we have $\calP_H(\overline{v_i}) = \calP_{H_i}(\overline{v_i})$ for each $i\in [m].$
\end{observation}

From the above observation, we can deduce the following lemma.

\begin{lemma}\label{lem:components-join}
    Let $H$ be a linear $k$-graph and $v\in V(H)$ be a vertex such that $H - v$ is a disjoint union of linear $k$-graphs $\{H_{(i, j)}: i\in [m], j\in [k-1]\}.$
    For each $i\in [m]$ and $j\in [k-1]$, let $e_1,\dots, e_d$ be the edges of $H$ containing $v,$ where $e_i = \{v, v_{(i, 1)}, \dots, v_{(i, k-1)}\}$ and $e_i\cap V(H_{(i, j)}) = \{v_{(i, j)}\}.$
    Then we have the following.
    $$\calP_H(\overline{v}) = \frac{1}{1 + \sum_{i\in [m]}\prod_{j\in [k-1]} \calP_{H_{(i, j)}}\left(\overline{v_{(i,j)}}\right)}.$$
\end{lemma}

\begin{proof}[Proof of \Cref{lem:components-join}]
     We do the same calculation that we presented in \Cref{sec:intro}. Since for each $i\in [m]$, $\calP_{H}(e_i \in M) = \calP_{H}\left(\overline{v}, \overline{v_{(i, 1)}}, \dots, \overline{v_{(i, k-1)}} \right)$, we obtain the following.

     \begin{equation}\label{eq:tree-1}
         \calP_{H}(\overline{v}) + \sum_{i\in [m]} \calP_{H}\left(\overline{v}, \overline{v_{(i, 1)}}, \dots, \overline{v_{(i, k-1)}} \right) = 1.
     \end{equation}
        
    By dividing each side of \eqref{eq:tree-1} by $\calP_{H}(\overline{v})$ and then take inverse, we obtain 

    \begin{equation}\label{eq:tree-2}
        \calP_{H}(\overline{v}) = \frac{1}{1 + \sum_{i\in [m]} \calP_{H}\left(\overline{v_{(i, 1)}}, \dots, \overline{v_{(i, k-1)}} \middle| \overline{v} \right)}.
    \end{equation}

    We note that for each $i\in [m]$, $\calP_{H}\left(\overline{v_{(i, 1)}}, \dots, \overline{v_{(i, k-1)}} \middle| \overline{v} \right)$ is same as $\calP_{H - v}\left(\overline{v_{(i, 1)}}, \dots, \overline{v_{(i, k-1)}}\right).$ Since $H - v$ is a disjoint union of $H_{(i, j)}$, by \Cref{obs:independent}, the following holds for each $i\in [m].$

    \begin{equation}\label{eq:tree-3}
        \calP_{H - v}\left(\overline{v_{(i, 1)}}, \dots, \overline{v_{(i, k-1)}}\right) = \prod_{j\in [k-1]}\calP_{H_{(i, j)}}\left( \overline{v_{(i, j)}} \right).
    \end{equation}

    By combining \eqref{eq:tree-2} and \eqref{eq:tree-3}, we obtain the desired formula. This completes the proof.
\end{proof}

We are now ready to prove \Cref{lem:counterexample}.

\begin{proof}[Proof of \Cref{lem:counterexample}]
    Let $F$ be a $d$-extendable linear $k$-graph obtained by \Cref{lem:base-existence}.

    \begin{claim}\label{clm:iterative}
        Let $F$ be a $d$-extendable linear $k$-graph. Then the following holds.
        $$\calP_{S_d(F)}\left(\overline{\mathrm{head}(S_d(F))}\right) = g_d\left(\calP_{F}\left(\overline{\mathrm{head}(F)}\right)\right).$$
    \end{claim}

    \begin{claimproof}[Proof of \Cref{clm:iterative}]
       If we remove the head of $S_d(F)$ from $S_d(F)$, then the remaining hypergraph is a disjoint union of copies of $F$, and by the definition of the operation $S_d(\cdot)$, we can apply \Cref{lem:components-join} to $S_d(F)$ and the head of $S_d(F).$ Since the head of $S_d(F)$ has degree $d-1$, by \Cref{lem:components-join}, the desired equality holds. This proves the claim.
    \end{claimproof}

    The next claim states that the probability that a random matching does not contain the head of $F$ cannot be the value $\alpha_d$, which is one of the fixed points of the function $f_d$ that is highly unstable.
    
    \begin{claim}\label{clm:irrational}
        $\calP_F\left(\overline{\mathrm{head}(F)}\right) \neq \alpha_d.$
    \end{claim}

    \begin{claimproof}[Proof of \Cref{clm:irrational}]
        The number $\calP_F\left(\overline{\mathrm{head}(F)}\right)$ can be computed by dividing the number of matchings in $F - v$ by the number of matchings in $F$, so the number $\calP_F\left(\overline{\mathrm{head}(F)}\right)$ is a rational number. Hence, to finish the proof, it suffices to show that $\alpha_d$ is irrational. Assume $\alpha_d$ is a rational number and it is $\frac{n}{m}$ for some integer $n, m$ which are coprime to each other. Then we have $\frac{n}{m} + \frac{(d - 1)n^k}{m^k} = 1.$ Thus $(d - 1)n^k = m^{k-1}(m - n).$ Note that the left-hand side is divisible by $n$ while the right-hand side is not divisible by $n$ because $n$ and $m$ are coprime, a contradiction. This proves the claim.
    \end{claimproof}

    By \Cref{clm:irrational}, we know that $\calP_{F}\left(\overline{\mathrm{head}(F)}\right) \neq \alpha_d.$ If $\calP_{F}\left(\overline{\mathrm{head}(F)}\right) > \alpha_d$, then we set $F_0 = F.$ Otherwise, we set $F_0 = S_d(F).$ Then by \Cref{clm:iterative}, we have the following.
    $$\alpha_d < \calP_{F_0}\left(\overline{\mathrm{head}(F_0)}\right) \leq 1.$$

    Let $p_{i} = \calP_{S_d^{(i)}(F_0)}\left(\overline{\mathrm{head}(S_d^{(i)})(F_0)}\right)$ for each $i\in \mathbb{Z}_{\geq 0}.$ Then by \Cref{clm:iterative}, we have $p_{i+1} = g_d(p_i)$ for each $i\in \mathbb{Z}_{\geq 0}$ and $\alpha_d < p_0 \leq 1.$ Since $f_d(x) = g_d(g_d(x))$, for each $i\in \mathbb{Z}_{\geq 0}$, we have $p_{i+2} = f_d(p_i).$ Then by \Cref{lem:attractor}, the sequence $\{p_{2i}\}_{i\in \mathbb{Z}_{\geq 0}}$ converges to $\beta_d$ as $i$ goes to infinity. Similarly, since $0 < p_1 = g_d(p_0) < \alpha_d$ holds, the sequence $\{p_{2i-1}\}_{i\in \mathbb{Z}_{\geq 0}}$ converges to $\gamma_d$ as $i$ goes to infinity. This completes the proof.
\end{proof}

We have collected all ingredients to prove \Cref{thm:main}

\begin{proof}[Proof of \Cref{thm:main}]
    Let $d \equiv 1 \pmod{k}$ be a sufficiently large integer and $\delta$ be a sufficiently smaller positive real number than $\ve.$
    By \Cref{lem:counterexample}, there is a $d$-extendable linear $k$-graph $F_0$ and $\ell$ such that 
    \begin{equation}\label{eq:odd-level}
        \calP_{S_d^{(2\ell + 1)}(F_0)}\left(\overline{\mathrm{head}(S_d^{(2\ell+1)}(F_0))} \right) < \frac{1 + \delta}{d}
    \end{equation}
    and
    \begin{equation}\label{eq:even-level}
        \calP_{S_d^{(2\ell)}(F_0)}\left(\overline{\mathrm{head}(S_d^{(2\ell+1)}(F_0))} \right) > 1 - \frac{1 + \delta}{d^{k-2}}.
    \end{equation}
    
    Let $H_0$ be a hypergraph $S_d^{(2\ell + 1)}(F_0).$ Now we consider $H' = \bigcup_{i\in [d], j\in [k-1]} H_{(i, j)}$ which is a disjoint union of $(k-1)d$ copies of $H_0.$ For each $i\in [d]$ and $j\in [k-1]$, let $v_{(i, j)}$ be the head of $H_{(i, j)}.$ We now introduce a new vertex $v$ that is disjoint from $V(H').$ Define a $k$-graph $H$ as $V(H) = V(H') \cup \{v\}$ and $E(H) = E(H') \cup \bigcup_{i\in [d-1]} \{v, v_{(i, 1)}, \dots, v_{(i, k-1)}\}.$ Since $H_0$ is $d$-extendable linear $k$-graph, we observe that $H$ is a $d$-regular linear $k$-graph.

    We now claim that $H$, $v$, $v_{(1, 1)}$ would be the hypergraph and vertices that satisfy the conclusion of \Cref{thm:main}. Then \Cref{lem:components-join} and \eqref{eq:odd-level} implies 
    $$\calP_H(\overline{v}) > \frac{1}{1 + d \left(\frac{1 + \delta}{d}\right)^{k-1}} > 1 - \frac{1 - \ve}{d^{k-2}}.$$

    Similarly, \Cref{lem:components-join} and \eqref{eq:even-level} implies 
    $$\calP_H\left(\overline{v_{(1, 1)}}\right) < \frac{1}{1 + (d - 1)\left(1 - \frac{1 - \delta}{d^{k-2}}\right)^{k-1} + 1 } < \frac{1 + \ve}{d + 1}.$$ This completes the proof.
\end{proof}


\section{Conflict-free walks and matching polynomials}\label{sec:define-tree}

In this section, we generalize Godsil's result~\cite{godsil} into hypergraphs. We start with \emph{matching polynomials} for hypergraphs.
The matching polynomial of an $n$-vertex $k$-graph $H$ is given by $$m_k(H, x) = \sum_{i = 0}^{\lfloor n/k \rfloor} (-1)^i p(H, i)x^{n-ki},$$ where $p(H, i)$ is the number of distinct matchings of size $i$ in $H$ where we let $p(H, 0) = 1.$

A similar notion of the matching polynomial of $H$ is the \emph{matching-generating polynomial} of $H$, which is given by $$q_k(H, x) = \sum_{i = 0}^{\lfloor n/k \rfloor} p(H, i)x^i.$$ 
We note that the following identity between the matching polynomials and matching-generating polynomials holds.
\begin{equation}\label{eq:identity}
    m_k(H, x) = x^n q_k(H, -x^k).
\end{equation}

The study of matching polynomials for hypergraphs was recently initiated, to get more information on it, we recommend to see~\cite{clark-cooper,su-kang-li-shan,wei-shan}.

As the purpose of this section is to extend the result of Godsil~\cite{godsil} into hypergraphs, we recall Godsil's notion of \emph{path-tree}. Let $G$ be a graph and $v$ be a vertex of $G.$ The path-tree $T = T(G, v)$ of $G$ rooted at $v$ is a tree such that the vertex set of $T$ is the paths in $G$ which start from $v$ and we join two distinct paths in $T(G, v)$ whenever one is the maximal sub-path of the another. The main theorem of \cite{godsil} is the following.

\begin{theorem}[Godsil~\cite{godsil}]\label{thm:godsil}
    Let $G$ be a graph and $v \in V(G).$ Then the following holds.
    $$\frac{m_2(G - v, x)}{m_2(G, x)} = \frac{m_2(T(G, v) - V, x)}{m_2(T(G, v), x)},$$
    where $V$ is an one-vertex path $(v)$ of $G.$
\end{theorem}

By combining \Cref{thm:godsil} and \eqref{eq:identity}, we obtain the following equality. 
\begin{equation}\label{eq:convert-to-tree}
    \calP_G(\overline{v}) = \calP_{T(G, v)}(\overline{V}).
\end{equation}
This implies that we can convert the host graph as a path-tree, so we can recursively compute the value of $\calP_G(\overline{v}).$

Our goal for this section is to construct a hypergraph analog of a path-tree and establish a similar identity to \Cref{thm:godsil} for hypergraphs. Let $H$ be a hypergraph and $v_1, \dots, v_n \in V(H)$ be distinct vertices and $e_1, \dots, e_{n-1} \in E(H)$ be distinct edges of $H.$ We say $W = (v_1, e_1, v_2, \dots, v_{n-1}, e_{n-1}, v_n)$ is a \emph{Berge-path} starts from $v_1$ and ends at $v_n$ if $v_i, v_{i+1} \in e_i$ for each $i\in [n-1].$ This notion was introduced by Berge~\cite{berge1973graphs}.
A $k$-graph $T$ is a \emph{$k$-uniform hypertree} if, for every pair of vertices $u$ and $v$, there is a unique Berge-path that starts from $u$ and ends at $v.$ Definitions of rooted $k$-uniform hypertree and its rooted sub-hypertree are the obvious hypergraph versions of rooted trees and subtrees in graph theory. 

Berge-paths usually play the role of a hypergraph version of paths in several topics related to hypergraphs. However, to extend the notion of the path into an appropriate hypergraph analog that matches \Cref{thm:godsil}, we need a new type of Berge-path, which we call \emph{conflict-free walk}.

\begin{definition}\label{def:conflict-free}
    Let $H$ be a $k$-graph and $\prec$ be an arbitrary linear ordering of $V(H).$ Let $\ell \geq 1$ be an integer, $v_0, \dots, v_{\ell}\in V(H)$ and ${v_{i-1}, v_i}\subset e_i \in E(H)$ for each $i\in [\ell].$
    Let $e_i = \{v_{i-1}, u_{(i, 1)}, \dots, u_{(i, k-2)}, v_i\}$ and $C_i = \{u_{(i, j)}: u_{(i, j)} \prec v_i, j\in [k-2]\} \cup \{v_{i-1}\}$ for each $i\in [\ell].$
    We say a Berge-path $W = (v_0, e_1, v_1, e_2, v_3, \dots, v_{\ell - 1}, e_{\ell}, v_{\ell})$ is a \emph{conflict-free walk} starts from $v_0$ and ends at $v_{\ell}$ if for each $2\leq i\leq \ell$, the edge $e_i$ is disjoint with the set $\bigcup_{j=1}^{i-1} C_j.$ We also consider an one-vertex Berge-path $V = (v)$ as a conflict-free walk starts from $v$ and ends at $v.$
\end{definition}

We now define \emph{$k$-walk-tree}, which is a hypergraph analog of the path-tree.

\begin{definition}\label{def:walk-tree}
    Let $H$ be a $k$-graph, $v\in V(H),$ and we fix a linear ordering $\prec$ of $V(H).$ We define a $k$-walk-tree $T(H, v)$ of $H$ rooted at $v$ as follows. The vertex set of $T(H, v)$ is the set of conflict-free walks that start from $v.$ For the edge set, we join $k$ conflict-free walks $W_0, W_1, \dots, W_{k-1}$ as a hyperedge of $T(H,v)$ if there is an edge $e = \{u_0, \dots, u_{k-1}\} \in E(H)$ such that $W_0$ is a conflict-free walk starts from $v$ and ends at $u_0$ and for each $i\in [k-1]$, the conflict-free walk $W_i$ ends at $u_i$ and it is obtained from $W_0$ by adding an edge $e.$
\end{definition}

The following observation directly follows from \Cref{def:conflict-free}.

\begin{observation}\label{obs:structure-walk-tree}
    Let $H$ be a $k$-graph, $v\in V(H),$ and we fix a linear ordering $\prec$ of $V(H).$ Let $e_1, \dots, e_m \in E(H)$ be the set of edges that contain $v$ such that $e_i = \{v, u_{(i, 1)}, \dots, u_{(i, k-1)}\}$ and $u_{(i, 1)} \prec \cdots \prec u_{(i, k-1)}$ for each $i\in [m].$ Then the following holds.
    \begin{enumerate}
        \item[$(i)$] $T(H, v)$ is a $k$-uniform hypertree,
        \item[$(ii)$] if $H$ is a hypertree, then $T(H, v)$ is isomorphic to $H$, where $V = (v)$ in $T(H, v)$ is mapped to $v$ in $H$,
        \item[$(iii)$] $T(H, v)$ can be obtained from the collection of disjoint union of hypertrees $$\Bigl\{\bigcup_{j = 1}^{k-1} T(H - \{v, u_{(i, 1)}, \dots, u_{(i, j-1)}\}, u_{(i, j)}): i\in [m]\Bigr\}$$ by adding a hyperedge $\{v, (u_{(i, 1)}), \dots, (u_{(i, k-1)})\}$ for each $i\in [m]$, where $(u_{(i, j)})$ is a root of $T(H - \{v, u_{(i, 1)}, \dots, u_{(i, j-1)}\}, u_{(i, j)})$ for each $i\in [m]$ and $j\in [k-1].$
    \end{enumerate}
\end{observation}

We note that if $k = 2$, then conflict-free walks are paths and $k$-walk-trees are path-trees.

We now claim that the $k$-walk-tree is an appropriate generalization of Godsil's path-tree.

\begin{theorem}\label{thm:walk-tree}
    Let $H$ be a $k$-graph, $v \in V(H),$ and we fix a linear ordering of $V(H).$ Then the following holds.
    $$\frac{m_k(H - v, x)}{m_k(H, x)} = \frac{m_k(T(H, v) - V, x)}{m_k(T(H, v), x)},$$ where $V$ is an one-vertex conflict-free walk $(v).$
\end{theorem}

To prove \Cref{thm:walk-tree}, we collect the following two lemmas. 

\begin{lemma}\label{lem:disconnected}
    Let $H$ be a $k$-graph which is a disjoint union of two vertex-disjoint $k$-graphs $H_1$ and $H_2.$ Then the following holds.
    $$m_k(H, x) = m_k(H_1, x)m_k(H_2, x).$$
\end{lemma}

The proof of \Cref{lem:disconnected} can be directly obtained from the definition of the matching polynomial.

\begin{lemma}[Su, Kang, Li, and Shan~\cite{su-kang-li-shan}]\label{lem:su-kang-li-shan}
    Let $H$ be a $k$-graph and $v\in V(H).$  Let $e_1, \dots, e_m \in E(H)$ be the set of edges containing $v.$ Denote $V(e_i)$ as the set of vertices of $e_i$ for each $i\in [m].$ Then we have the following identity.
    $$\frac{m_k(H, x)}{m_k(H - v, x)} = x - \sum_{i\in [m]} \frac{m_k(H - V(e_i)), x}{m_k(H - v, x)}.$$
\end{lemma}

\begin{proof}[Proof of \Cref{thm:walk-tree}]
    We use induction on the number of vertices of $H.$ If $V(H) \leq k$, then $H$ is edgeless or it has only one edge. Then by \Cref{obs:structure-walk-tree} $(ii),$ the theorem is true. Now we assume $V(H) > k.$ Let $e_1, \dots, e_m \in E(H)$ be the set of edges that contain $v$ such that $e_i = \{v, u_{(i, 1)}, \dots, u_{(i, k-1)}\}$ and $u_{(i, 1)} \prec \cdots \prec u_{(i, k-1)}$ for each $i\in [m].$ Then we have the following identities.
    \allowdisplaybreaks
    \begin{align*}
        \frac{m_k(H, x)}{m_k(H - v, x)} &= x - \sum_{i\in [m]} \frac{m_k(H - V(e_i), x)}{m_k(H - v, x)} &&\text{(by \Cref{lem:su-kang-li-shan})}\\
        &= x - \sum_{i\in [m]}\prod_{j\in [k-1]} \frac{m_k(H - \{v, u_{(i, 1)}, \dots, u_{(i, j)}\}, x)}{m_k(H - \{v, u_{(i, 1)}, \dots, u_{(i, j-1)}\}, x)}\\
        &= x - \sum_{i\in [m]}\prod_{j\in [k-1]} \frac{m_k(T(H - \{v, u_{(i, 1)}, \dots, u_{(i, j-1)}\}, u_{(i, j)}) - U_{(i, j)}, x)}{m_k(T(H - \{v, u_{(i, 1)}, \dots, u_{(i, j-1)}\}, U_{(i, j)}), x)} &&\text{(by induction)}\\
        &= x - \sum_{i\in [m]}\prod_{j\in [k-1]} \frac{m_k(T(H, v) - \{V, U_{(i, 1)}, \dots, U_{(i, j-1)}\}, x)}{m_k(T(H, v) - \{V, U_{(i, 1)}, \dots, U_{(i, j)}\}, x)} \\ &\text{\quad (by \Cref{obs:structure-walk-tree} and \Cref{lem:disconnected})}\\
        &= x - \sum_{i\in [m]} \frac{m_k(T(H, v) - V(E_i), x)}{m_k(T(H, v) - v, x)}\\
        &= \frac{m_k(T(H, v), x)}{m_k(T(H, v) - v, x)}. &&\text{(by \Cref{lem:su-kang-li-shan})}
    \end{align*}
    Here, $V = (v)$ and $U_{(i, j)}$ is the natural conflict-free walk starts from $v$ that correspond to the vertex $u_{(i, j)}$ for each $i\in [m]$ and $j\in [k-1].$ Similarly, $E_i$ is the edge of $T(H, v)$ that naturally corresponding with $e_i$ for each $i\in [m].$ 
    This completes the proof.
\end{proof}

By \eqref{eq:identity}, we also obtain the following from \Cref{thm:walk-tree}. $$\frac{q_k(H - v, x)}{q_k(H, x)} = \frac{q_k(T(H, v) - V, x)}{q_k(T(H, v), x)},$$ where $V = (v).$ Thus we deduce the following corollary.

\begin{corollary}\label{cor:probability-converting}
    Let $H$ be a $k$-graph and $v\in V(H)$, $V = (v).$ Then we have the following equality. 
    $$\calP_H(\overline{v}) = \calP_{T(H, v)}(\overline{V}).$$
\end{corollary}


\section{Proof of \Cref{thm:main-tight}}\label{sec:proof-main-tight}

In this section, we prove \Cref{thm:main-tight} by using \Cref{cor:probability-converting}. We start with the following proposition.

\begin{proposition}\label{prop:recursive}
    Let $H$ be a $k$-graph, $v\in V(H),$ and we fix a linear ordering $\prec$ of $V(H).$ Let $e_1, \dots, e_m \in E(H)$ be the set of edges that contain $v$ such that $e_i = \{v, u_{(i, 1)}, \dots, u_{(i, k-1)}\}$ and $u_{(i, 1)} \prec \cdots \prec u_{(i, k-1)}$ for each $i\in [m].$ Then the following holds.

    \begin{equation*}
        \calP_H(\overline{v}) = \left(1 + \sum_{i\in [m]}\prod_{j\in [k-1]}\calP_{H-\{v, u_{(i, 1)}, \dots, u_{(i, j-1)}\}}\left(\overline{u_{(i, j)}}\right)\right)^{-1}
    \end{equation*}
\end{proposition}

\begin{proof}[Proof of \Cref{prop:recursive}]
    We observe the probability of the random matching containing an edge $e_i$ is the same as the probability that the random matching $M$ does not contain all the vertices of $V(e_i)$ for each $i\in [m].$ Thus we have the following.

    \begin{equation*}
        \calP_H(\overline{v}) + \sum_{i\in [m]} \calP_H(e_i\in E(M)) = \calP_H(\overline{v}) + \sum_{i\in [m]} \calP_H\left(\overline{v}, \overline{u_{(i, 1)}}, \dots, \overline{u_{(i, k-1)}}\right) = 1.
    \end{equation*}

    By dividing $\calP_H(\overline{v})$, we obtain
    \begin{equation}\label{eq:1}
        1 + \sum_{i\in [m]} \calP_H\left(\overline{u_{(i, 1)}}, \dots, \overline{u_{(i, k-1)}}\middle| \overline{v}\right) = \calP_H(\overline{v})^{-1}.    
    \end{equation}
    
    By the chain rule of the conditional probability, we have 
    \begin{equation}\label{eq:2}
        \calP_H\left(\overline{u_{(i, 1)}}, \dots, \overline{u_{(i, k-1)}}\middle| \overline{v}\right) = \prod_{j\in [k-1]}\calP_H\left(\overline{u_{(i, j)}}\middle| \overline{v}, \overline{u_{(i, 1)}}, \dots, \overline{u_{(i, j-1)}}\right)    
    \end{equation}
    for each $i\in [m].$
    Finally, we note that 
    \begin{equation}\label{eq:3}
        \calP_H\left(\overline{u_{(i, j)}}\middle| \overline{v}, \overline{u_{(i, 1)}}, \dots, \overline{u_{(i, j-1)}}\right) = \calP_{H-\{v, u_{(i, 1)}, \dots, u_{(i, j-1)}\}}\left(\overline{u_{(i, j)}}\right)
    \end{equation}
    holds for each $i\in [m]$ and $j\in [k-1].$

    By combining \eqref{eq:1}--\eqref{eq:3}, we obtain the desired equality.
    This completes the proof.
\end{proof}

From now, we fix our $k$-graph $H$, $v\in V(H)$, and the linear order $\prec$ on $V(H).$ We denote by $V$ the one-vertex conflict-free walk $(v)$ and $T$ the $k$-walk-tree $T(H, v).$ For a conflict-free walk $U$ that starts from $v$, we write $T(U)$ for the sub-hypertree of $T$ rooted at $U.$

\begin{lemma}\label{lem:main-lemma}
    Let $E_1, \dots, E_m$ be the set of edges that contain $V$ in $T.$ Let $E_i = \{V, U_{(i, 1)}, \dots, U_{(i, k-1)}\}$ for each $i\in [m].$ Then the following holds.
    \begin{equation*}
        \calP_H(\overline{v}) = \calP_T(\overline{V}) = \frac{1}{1 + \sum_{i\in [m]}\prod_{j\in [k-1]}\calP_{T(U_{(i, j)})}\left(\overline{U_{(i, j)}}\right)}.
    \end{equation*}
\end{lemma}

\begin{proof}[Proof of \Cref{lem:main-lemma}]
    Let $e_1, \dots, e_m \in E(H)$ be the set of edges that contain $v$ such that for each $i\in [m]$ and $j\in [k-1]$, the edge $e_i = \{v, u_{(i, 1)}, \dots, u_{(i, k-1)}\}$ and $u_{(i, 1)} \prec \cdots \prec u_{(i, k-1)}$, and $u_{(i, j)}$ is the conflict-free walk $(v, e_i, U_{(i, j)}).$
    
    By \Cref{cor:probability-converting} and \Cref{prop:recursive}, we have the following.
    \begin{equation*}
        \calP_H(\overline{v}) = \calP_T(\overline{V}) = \left(1 + \sum_{i\in [m]}\prod_{j\in [k-1]}\calP_{T(H-\{v, u_{(i, 1)}, \dots, u_{(i, j-1)}\}, u_{(i, j)})}\left(\overline{U'_{(i, j)}}\right)\right)^{-1},
    \end{equation*}
    where $U'_{(i, j)}$ is the conflict-free walk $(U_{(i, j)})$ in the $k$-walk-tree $T(H-\{v, u_{(i, 1)}, \dots, u_{(i, j-1)}\}, u_{(i, j)}).$ From the definition of the conflict-free walk and $k$-walk-tree, we observe that for each $i\in [m]$ and $j\in [k-1]$, the two hypertrees $T(U_{(i, j)})$ and $T(H-\{v, u_{(i, 1)}, \dots, u_{(i, j-1)}\}, u_{(i, j)})$ are isomorphic, where $U_i^{j}$ is mapped to $U'_{(i, j)}.$
    This completes the proof.
\end{proof}

We remark that by \Cref{obs:structure-walk-tree} $(ii),$ the $k$-walk-tree of hypertree is isomorphic to the original hypertree, thus we can recursively apply \Cref{lem:main-lemma} by starting from the leaves of the $k$-walk-tree to estimate $\calP_H(\overline{v}).$

\begin{proof}[Proof of \Cref{thm:main-tight}]
    Let $d$ be a sufficiently large non-negative integer.
    We apply \Cref{lem:main-lemma} to $H$, $v$, and $T.$ Since $H$ is a $d$-regular linear $k$-graph, by expanding the formula in the statement of \Cref{lem:main-lemma} up to the second level of hypertree $T$, we obtain the following.
    \begin{equation}\label{eq:final}
        \calP_H(\overline{v}) = \left(1 + \sum_{i\in [d]}\prod_{j\in [k-1]} \left(1 + \sum_{i'\in [d-1]}\prod_{j'\in [k-1]} \calP_{T(U_{i, i'}^{j, j'})}\left( \overline{U_{i, i'}^{j, j'}}\right) \right)^{-1}\right)^{-1},
    \end{equation}
    where for each $i\in [d], i'\in [d-1]$ and $j, j'\in [k-1]$, the walk $U_{i, i'}^{j, j'}$ is a conflict-free walk of length two starts from $v$.

    We observe that the right-hand side of \eqref{eq:final} is smaller than equal to the value obtained by replacing all the numbers $\calP_{T(U_{i, i'}^{j, j'})}\left( \overline{U_{i, i'}^{j, j'}}\right)$ by $1.$ Thus we obtain $$\calP_H(\overline{v}) \leq \left(1 + \sum_{i\in [d]}\prod_{j\in [k-1]} d^{-1} \right)^{-1} < 1 - \frac{1 - \ve}{d^{k-2}}.$$ 
    Similarly, the right-hand side of \eqref{eq:final} is greater than equal to the value obtained by replacing all the numbers $\calP_{T(U_{i, i'}^{j, j'})}\left( \overline{U_{i, i'}^{j, j'}}\right)$ by $0.$ Thus the following holds.
    $$\calP_H(\overline{v}) \geq \left(1 + \sum_{i\in [d]}\prod_{j\in [k-1]} 1 \right)^{-1} \geq \frac{1}{d+1}.$$ 
    This completes the proof.
\end{proof}


\section{Concluding remarks}\label{sec:concluding}
We note that in the proof of \Cref{thm:main-tight}, the inequality $\calP_{H}(\overline{v}) \geq \frac{1}{d+1}$ do not need the assumption that $d$ is sufficiently large and $k > 2.$ Thus the following corollary holds from the proof.

\begin{corollary}\label{cor:average}
    If $H$ is an $n$-vertex $d$-regular linear $k$-graph with $k \geq 2$ and $d \geq 1,$ then the average size of matchings in $H$ is at most $\left(1 - \frac{1}{d+1}\right)\frac{n}{k}.$ 
\end{corollary}

Although, \Cref{thm:main} shows that $\max\{\calP_H(\overline{v}) : v\in V(H)\}$ may be close to $1,$ it is stil possible that $\sum_{v\in V(H)} \calP_H(\overline{v}) = o_d(n).$ We believe this is true.

In this paper, we disprove \Cref{conj:kahn} for all $k > 2$ and provide a sharp bound for $\calP_H(\overline{v})$ for $d$-regular linear $k$-graph. The construction for the counterexample uses a limit of dynamical systems, thus the number of vertices in our hypergraph is very large compared with $d.$ The proof highly relies on the fact that $d$ is a fixed constant, so we can iteratively expand our hypergraph until we obtain our desired properties. Then what if $d$ also grows with the order of the hypergraph? In this case, we suspect that \Cref{conj:kahn} would be true for such values of $d.$

\begin{conjecture}\label{conj:polynomial}
    Let $k \geq 3$ be a positive integer and $0 < \delta, \ve < 1$ be a positive real number. Then there is $n_0$ such that the following is true for all $n \geq n_0.$
    Let $H$ be an $n$-vertex $k$-graph that satisfies the following. For all $v\in V(H)$, the degree of $v$ lies between $(1 - n^{-\delta})d$ and $(1 + n^{-\delta})d$ for some $d > n^{\ve}$, and the maximum codegree of $H$ is bounded above by $n^{-\delta} d.$ Then for all $v\in V(H),$ we have
    $$\calP_H(\overline{v}) = (1 + o_n(1))d^{-1/k}.$$
\end{conjecture}

If \Cref{conj:polynomial} is true, there are several interesting consequences. For instance, this implies that the average size of partial transversals in sufficiently large $n$-symbol Latin squares is $(1 - n^{-1/3} + o_n(1))n$, meaning that the average size of partial transversals hardly depends on the specific structure of the Latin squares. We can also obtain sharp estimates for average sizes of partial combinatorial designs in the hypergraphs if \Cref{conj:polynomial} is true.

\Cref{thm:walk-tree} generalize Godsil's celebrated identity~\cite{godsil} to hypergraphs. The matching polynomials and Godsil's path-trees have many applications in various areas such as statistical physics, we believe \Cref{thm:walk-tree} have the potential to generalize several results on matching polynomials on graphs into hypergraphs. 


\subsection*{Acknowledgement}
This research was performed during the author's visit to the Universit\"{a}t Hamburg. We thank the Universit\"{a}t Hamburg for its support, and hospitality, and for providing a great working environment. The author also thanks his advisor Jaehoon Kim for his helpful advice and encouragement.


\printbibliography

\end{document}